%% file: main.tex
\documentclass[journal,a4paper]{IEEEtran}

\usepackage[noadjust]{cite}

\usepackage[english]{babel}
\usepackage[final]{graphicx}

\usepackage[cmex10]{amsmath}
\usepackage{amsfonts}
\usepackage{amssymb}
\usepackage{amsthm}
\usepackage{verbatim}
\usepackage{mathbbol}	% improves mathbb for 1, R, C
\usepackage{algorithm}
\usepackage{algorithmic}

\usepackage{booktabs}

\usepackage[T1]{fontenc}

\newtheorem{theorem}{Theorem}

\newtheorem*{remark}{Remark}
\newtheorem*{definition}{Definition}
\newtheorem*{example}{Example}
\newtheorem{assumption}{Assumption}

\usepackage{enumerate}
\usepackage{array}

\usepackage{url}

\newlength{\IEEECOLWIDTH}
\usepackage{booktabs}

\def\TITLE{Hierarchical and Distributed Monitoring of Voltage Stability in Distribution Networks}
\def\TITLEnl{Hierarchical and Distributed Monitoring\\of Voltage Stability in Distribution Networks}
\def\AUTHORS{Liviu Aolaritei, Saverio Bolognani and Florian D\"orfler}

\usepackage[hyperindex=true, %
  pdftitle={\TITLE},%
  pdfauthor={\AUTHORS},%
  colorlinks=true,%
  pagebackref=false,%
  plainpages=false,%
  pdfpagelabels,%
  linkcolor=black,%
  citecolor=black,%
  filecolor=black,%
  urlcolor=black%
]{hyperref} 

\def\1{\mathbf{1}}
\def\0{\mathbf{0}}

\DeclareMathOperator{\realpart}{Re}

\def\jay{\textbf{j}}
\def\t{\top}

\usepackage[absolute,overlay]{textpos}

\begin{document}

\begin{textblock*}{\textwidth}(15mm,50mm) % {block width} (coords) 
\centering \bf Published on \emph{IEEE Transactions on Power Systems,} vol. 33, no. 6, pp. 6705-6714, Nov. 2018.\\\url{http://doi.org/10.1109/TPWRS.2018.2850448}
\end{textblock*}

\title{\TITLEnl}

\author{\AUTHORS% <-this % stops a space
\thanks{The authors are affiliated with the Automatic Control Laboratory,
	ETH Zurich, 8092 Zurich, Switzerland. Email: \texttt{\small	\{aliviu, bsaverio, dorfler\}@ethz.ch.}}%
\thanks{This research is supported by ETH Zurich funds and the SNF Assistant Professor Energy Grant \# 160573.}%
\thanks{\textcopyright 2019 IEEE.  Personal use of this material is permitted.  Permission from IEEE must be obtained for all other uses, in any current or future media, including reprinting/republishing this material for advertising or promotional purposes, creating new collective works, for resale or redistribution to servers or lists, or reuse of any copyrighted component of this work in other works.}%
}

\maketitle

\begin{abstract}
We consider the problem of quantifying and assessing the steady-state voltage stability in radial distribution networks.
Our approach to the voltage stability problem is based on a local, approximate, and yet highly accurate characterization of the determinant of the Jacobian of the power flow equations parameterized according to the branch-flow model.
The proposed determinant approximation allows us to construct a voltage stability index that can be computed in a fully distributed or in a hierarchical fashion, resulting in a scalable approach to the assessment of steady-state voltage stability.
Finally, we provide upper bounds for the approximation error and we numerically validate the quality and the robustness of the proposed approximation with the IEEE 123-bus test feeder.
\end{abstract}

\begin{IEEEkeywords}
Voltage stability, distribution network, power flow Jacobian, power flow solvability, distributed algorithms.
\end{IEEEkeywords}

\section{Introduction}

\IEEEPARstart{F}{uture} electric power distribution grids are expected to host a larger amount of microgeneration, especially from intermittent and uncontrollable renewable sources, and to serve the higher power demand caused by ubiquitous penetration of plug-in electric vehicles.
To survive these radical changes, these grids are expected to become ``smart'', and therefore to be provided with online monitoring solutions, self-healing mechanisms, and enhanced flexibility in their operation.

One of the phenomena that occur when power flows reach or exceed the power transfer capacity of the grid, is the loss of \emph{long-term} voltage stability \cite{Loef1993,Cutsem1998}. 
Voltage instability, and ultimately voltage collapse, is a complex dynamical phenomenon that has its origins in the coupling between the nonlinearity of power flow equations and the dynamic response of the devices connected to the grid (generators, regulators, tap changers, and loads) \cite{Dobson1989,Chiang1990}.
The dynamic aspects of voltage collapse have been connected to bifurcations phenomena of the static nonlinear power flow equations, to give a quasi-static characterization of voltage stability (i.e., based on the solvability of power flow equations) in the seminal works \cite{Tamura1983,Sauer1990,Dobson2011}.
This fundamental connection has been verified on different analytical models, time domain simulations, and historical data \cite{Canizares1995}.

A natural characterization of the solvability of power flow equations involves the invertibility of their Jacobian \cite{Venikov1975}. 
Based on this idea, voltage stability can also be quantified, for example by evaluating its minimum singular value \cite{Tiranuchit1988} or smallest eigenvalue {\cite{Gao1992}}.
Many similar quantitative indices have been proposed, mostly for transmission grids. 
The resulting stability certificates can be tested via centralized algorithms that assess the distance from voltage collapse of a given operating state of the grid (which could be a measured state, the output of a state estimator, or the solution of a power flow solver). 
Many of these stability indices have been compared and contrasted based on computational complexity and accuracy in predicting voltage instability
\cite{Kessel1986,Canizares1996,Sinha2000,Glavic2011}.
With the exception of {\cite{Kessel1986}} (which however requires some additional assumptions on the generator voltages), these methods require a global knowledge of the system parameters, and they are not amenable to distributed implementation.

In this paper we focus on balanced and radial power distribution networks.
Their radiality and the absence of voltage-regulated buses allows us to adopt the branch flow model \cite{Baran1989,Farivar2013} of the network and to propose a voltage stability index that directly descends from an approximation of the determinant of the power flow Jacobian.
Our index is a physically intuitive generalization of the well-known two-bus case and its computation is extremely efficient, even in large-scale networks. The quality of the approximation can be precisely evaluated and quantified for mono-directional flows, although numerical simulations show that the latter is rather a technical assumption and not a limiting factor for the applicability of our proposed index.

We also discuss how the proposed index can be used in a scenario where a large-scale distribution grid is provided with a distributed sensing architecture. 
We embrace the challenge proposed in \cite{SimpsonPorcoTSG2016} and \cite{Aolaritei2017}, with respect to the derivation of scalable and distributed algorithms for the computation of a voltage stability index.
We show that the index proposed in this paper is suitable for hierarchical decomposition as well as efficient distributed computation, and requires limited information on the grid parameters and topology.

Few other voltage stability indices have been derived for power distribution networks.
The voltage stability index proposed in \cite{Wang2017} (and the variant in \cite{Sun2018}) is also based on the singularity of the power flow Jacobian. 
It is however a centralized method, and it requires the knowledge of the full impedance matrix of the grid and of phasorial measurements.

Explicit conditions for the solvability of bi-quadratic power flow equations at each line of the grid have been proposed in \cite{Chakravorty2001,Augugliaro2007,Eminoglu2007},
exploiting both radiality and the presence of only PQ nodes.
These local conditions can be used to infer global voltage stability indices for the distribution grid. 
Interestingly, these indices perform similarly to the index proposed in this paper, which instead is derived starting from a global solvability condition, and then decomposed into individual terms for each bus.
One advantage of our index is its explicit connection to the nonsingularity of the power flow Jacobian: some numerical algorithms (e.g., \cite{Dvijotham2015}) specifically rely on this piece of information, and the proposed index can be used to achieve significant reduction in the computation complexity (see Section~\ref{subsec:cencomp}).

Heuristic indices have been obtained by considering 2-bus equivalent models of both transmission \cite{Chebbo1992} and distribution \cite{Gubina1997} grids.
Also in this case, there is no clear connection between these local indices and a global metric of the grid's distance from voltage collapse.

Finally, our proposed approach contrasts with the methodologies that have been recently proposed to characterize those power demands that can be satisfied by a stable voltage profile \cite{Bolognani2016,SimpsonPorco2016,WangPaolone2017,Dvijotham2017,Dvijotham2018,NguyenTSG}.
The spirit of those works is fundamentally different: they provide a characterization of the set in the multidimensional parameter space (namely, in the power injection space) that corresponds to a unique stable solution of the power flow equations. 
The main merit of those methods is to do so \emph{without} solving or attempting to solve the power flow equations. 
They are typically conservative and inherently rely on a precise knowledge of the system model. 
In contrast, the scalar voltage stability index that we propose in this paper relies on state measurements, and therefore on a solution of the power flow equations, either via numerical solvers or through the physics of the grid (i.e. by performing online measurements).
It can assess the proximity to voltage collapse very accurately even in the presence of a significant parametric mismatch in the model. Moreover, it is not sensitive to the specific load distribution that brings the system close to voltage collapse, therefore simplifying its numerical interpretation in a practical setting.

The paper is structured as follows.
In Section \ref{sec:model} we present the distribution grid model that is adopted throughout the paper.
In Section~\ref{sec:analysis} we quickly recall the connection between singularity of the power flow Jacobian and voltage collapse, and we specialize this criterion for the specific grid model that we adopted.
The proposed voltage stability index is presented in Section~\ref{sec:margin},
while in Section~\ref{sec:compvsia} we describe its computational scalability, and how it can be computed in distributed and hierarchical communication architectures.
Finally, in \ref{sec:accuracy} we discuss the accuracy of the proposed index, 
and in Section \ref{sec:numerical} we present some numerical experiments to validate its effectiveness and its robustness.
Section~\ref{sec:conclusions} concludes the paper.	

\section{Distribution network model}
\label{sec:model}

Let $G = (N,E)$ be a directed tree representing a symmetric and balanced power distribution network, where each node in $N = \{0,1,...,n\}$ represents a bus, and each edge in $E$ represents a line. Note that $|E|=n$. A directed edge in $E$ is denoted by $(i,j)$ and means that $i$ is the parent of $j$. For each node $i$, let $\delta(i) \subseteq N$ denote the set of all its children.
Node $0$ represents the root of the tree and corresponds to the grid substation bus.
For each $i$ but the root $0$, let $\pi(i) \in N$ be its unique parent.

We now define the basic variables of interest.
For each $(i,j) \in E$ let $\ell_{ij}$ be the squared magnitude of the complex current from bus $i$ to bus $j$, and $s_{ij} = p_{ij} + \jay q_{ij}$ be the sending-end complex power from bus $i$ to bus $j$.
Let $z_{ij} = r_{ij} + \jay x_{ij}$ be the complex impedance on the line $(i,j)$.
For each node $i$, let $v_i$ be the magnitude squared of the complex nodal  voltage, and $s_i = p_i + \jay q_i$ be the net complex power demand (load minus generation). 

We adopt the branch flow formulation of the power flow equations in a radial grid, as proposed in \cite{Baran1989,Farivar2013}:
\begin{align*}
& p_j = p_{\pi(j)j} - r_{\pi(j)j}l_{\pi(j)j} - \sum\limits_{k \in \delta(j)}p_{jk}, && \forall j \in N\\
&q_j = q_{\pi(j)j} - x_{\pi(j)j}l_{\pi(j)j} - \sum\limits_{k \in \delta(j)}q_{jk}, && \forall j \in N\\
&v_j = v_i - 2(r_{ij}p_{ij} + x_{ij}q_{ij}) + (r_{ij}^2 + x_{ij}^2)\ell_{ij}, && \forall (i,j) \in E \\
&v_i \ell_{ij} = p_{ij}^2 + q_{ij}^2, && \forall (i,j) \in E.
\label{eq:bfm}
\end{align*}

To write the same equations in vector form, we first define the vectors $p$, $q$, and $v$, obtained by stacking the scalars $p_i$, $q_i$, and $v_i$, respectively, for $i \in N$.
Similarly we define $\overline{p}$, $\overline{q}$, $\ell$, $r$, and $x$, as the vectors obtained by stacking the scalars $p_{ij}$, $q_{ij}$, $\ell_{ij}$, $r_{ij}$, and $x_{ij}$, respectively, for $(i,j) \in E$.

In the following, we make use of the compact notation $[x]$, where $x \in \mathbb{R}^n$, to indicate the $n\times n$ matrix that has the elements of $x$ on the diagonal, and zeros everywhere else. Moreover, we use the notation $\mathbf{1}$ for the all-ones vector and $\0$ for the zero matrix of appropriate dimensions.

We define two $(0,1)$-matrices $\Pi$ and $\Delta$, where $\Pi$ $\in$ $\mathbb{R}^{n+1 \times n}$ is the matrix which selects for each row $j$ the branch $(i,j)$, where $i = \pi(j)$, and $\Delta \in \mathbb{R}^{n+1 \times n}$ is the matrix which selects for each row $i$ the branches $(i,j)$, where $j \in \delta(i)$.
Notice that $A := \Delta-\Pi$ is the incidence matrix of the graph \cite{Diestel2016}.

The branch flow equations in vector form are
\begin{equation}
\begin{split}
p &= \Pi \big(\overline{p} - [r]\ell \big) - \Delta \overline{p} \\
q &= \Pi \big(\overline{q} - [x]\ell \big) - \Delta \overline{q} \\
\Pi^\t v &= \Delta^\t v - 2 \big( [r]\overline{p} + [x]\overline{q} \big) + \big([r]^2 + [x]^2 \big) \ell \\
\left[\Delta^\t v\right] \ell &= \left[ \overline{p}\right]\overline{p} + \left[\overline{q}\right]\overline{q}.
\end{split}
\label{eq:branchflowmodelvector}
\end{equation}

We model the node $0$ as a slack bus, in which $v_0$ is imposed ($v_0 = 1$ p.u.) and all the other nodes as PQ buses, in which the complex power demand (active and reactive powers) is imposed and does not depend on the bus voltage.
Therefore, the $2n+1$ quantities $(v_0, p_1, \ldots, p_n, q_1, \ldots, q_n)$ are to be interpreted as parameters, and the branch flow model consists of $4n+2$ equations in the $4n+2$ state variables $(\overline{p},\overline{q},\ell,v_1, \ldots, v_n,p_0,q_0)$.

\section{Voltage stability analysis}
\label{sec:analysis}

From the perspective of voltage stability, we define a \emph{loadability limit} of the power system as a critical operating point of the grid (in terms of nodal power demands), where the power transfer reaches a maximum value, after which the power flow equations have no solution.
There are infinitely many loadability limits, corresponding to different demand configurations.
Ideally, the power system will operate far away from these points, with a sufficient safety margin.
On the other hand, the \emph{flat voltage solution} (of the power flow equations) is the operating point of the grid where $v_i = 1$ for all $i$, $p=q=\0$, and $\overline{p}=\overline{q}=\ell=\0$.
This point is voltage stable, all voltages are equal to the nominal voltage, and the power system typically operates relatively close to it \cite{Cutsem1998}.

In the following, we recall and formalize the standard reasoning that allows to characterize loadability limits via conditions on the Jacobian of the power flow equations, and we specialize those results for the branch flow model that we have adopted.

\subsection{Characterization of the voltage stability region}

Based on the discussion at the end of Section~\ref{sec:model}, consider the two vectors $u = \left[\overline{p}^T, \overline{q}^T, \ell^T, v_1, \ldots, v_n, p_0, q_0\right]^T \in \mathbb{R}^{4n+2}$ and $\xi = \left[v_0, p_1, \ldots, p_n, q_1, \ldots, q_n\right]^T \in \mathbb{R}^{2n+1}$ 
corresponding to the state variables and the nodal parameters, respectively.
Then the branch flow model \eqref{eq:branchflowmodelvector} can be expressed in implicit form as
\begin{equation*}
	\varphi(u, \xi) = \0.
\end{equation*}

A loadability limit is formally defined as the maximum of a scalar function $\gamma(\xi)$ (to be interpreted as a measure of the total power transferred to the loads), constrained to the set $\varphi(u, \xi) = \0$ (the power flow equations), i.e., 
\begin{align*}
	\max\limits_{u, \xi} \quad & \gamma(\xi)\\
	\text{subject to} \quad &  \varphi(u, \xi) = \0.
\end{align*}

From direct application of the KKT optimality conditions \cite{Bertsekas2016}, it results that in a loadability limit the \emph{power flow Jacobian} $\varphi_u = \frac{\partial \varphi}{\partial u}$ becomes singular, i.e., $\det ( \varphi_u ) = 0$
(for details, see Chapter 7 in \cite{Cutsem1998}). 
Based on this, we adopt the following standard characterization for voltage stability of the grid.

\begin{definition} (Voltage stability region)
The voltage stability region of a power distribution network with one slack bus and $n$ PQ buses is the open connected set of power flow solutions that contains the flat voltage solution and where 
\begin{equation}
\det( \varphi_u ) \neq 0.
\label{eq:detpositive}
\end{equation}
\end{definition}

The assessment of voltage stability (and of the distance from voltage collapse) therefore requires the computation of the power flow Jacobian $\varphi_u$. 
In the next subsection, we show how this can be done under our modeling assumptions.

\subsection{The power flow Jacobian in the branch flow model}

When the branch flow model is adopted, $\varphi_u$ takes the form
\begin{equation}
	\varphi_u = 
	\begin{bmatrix}
		-A & \mathbf{0} & -\Pi[r] & \mathbf{0} & -\mathbf{e}_1 & \mathbf{0}\\
		\mathbf{0} & -A & -\Pi[x] & \mathbf{0} & \mathbf{0} & -\mathbf{e}_1\\
		-2[r] & -2[x] & [r]^2 + [x]^2 & A_2^\t & \mathbf{0} & \mathbf{0}\\
		2\left[\overline{p}\right] & 2\left[\overline{q}\right] & -\left[\Delta^\t v\right] & -\left[\ell\right] \Delta_2^\t & \mathbf{0} & \mathbf{0}
	\end{bmatrix}
\label{eq:pfj}
\end{equation}
where $\Delta_2$ and $A_2$ are the matrices obtained by removing the first row from $\Delta$ and $A$, respectively, and where $\mathbf{e}_1$ is the first canonical base vector.

We define the following $n \times n$ matrix, that we denote as the \emph{reduced power flow Jacobian}.
\begin{multline}
	 \varphi_u' = \left[ \Delta^\t v \right] + 2\left[\overline{p}\right]A_2^{-1}[r] + 2\left[\overline{q}\right]A_2^{-1}[x]  \\
- [\ell] \Delta_2^\t(A_2^T)^{-1}  \left( [r]^2 + 2 [r] A_2^{-1} [r] + [x]^2 + 2 [x] A_2^{-1} [x] \right)
\label{eq:rpfj}
\end{multline}

The following result shows the merits of the reduced power flow Jacobian.

\begin{theorem}
Consider the power flow Jacobian \eqref{eq:pfj} and the reduced power flow Jacobian \eqref{eq:rpfj}
of a power distribution network with one slack bus and $n$ PQ buses, described by the relaxed branch flow model.
We have:
\begin{itemize}
\item[i)] $\det( \varphi_u ) =  \det( \varphi_u' )$.
\item[ii)] $\det( \varphi_u' ) > 0$ in the voltage stability region.
\end{itemize}
\label{thm:voltagestabilityregion}
\end{theorem}
\begin{proof}
\emph{i)\ }
We first remove the last two columns of $\varphi_u$ together with the $1$-st and $(n+2)$-nd rows, obtaining a new matrix $\varphi_u^{*}$, whose determinant is equal to $(-1)^n \det(\varphi_u)$. 
Then, we apply the Schur complement twice to the matrix $\varphi_u^{*}$ and after some basic matrix manipulations obtain $\varphi_u'$, which satisfies $\det (\varphi_u') =  (-1)^n \det(\varphi_u^{*}$).
In each of the two Schur complements, the matrix (initially of dimensions $4n \times 4n$ and then of dimensions $2n \times 2n$) is divided into four blocks of equal dimensions, and the upper-left block is the invertible one.\\
\emph{ii)\ } 
In the flat voltage solution, $\varphi_u' = \left[ \Delta^T v \right]= \left[ \Delta^T \mathbf{1} \right] = \left[\mathbf{1}\right]$, therefore $\det(\varphi_u') = 1$.
Moreover, in a loadability limit, $\det(\varphi_u') = 0$. Since the determinant is a continuous function of the grid variables, in order to be in the voltage stability region, the determinant needs to remain positive.
\end{proof}

Theorem~\ref{thm:voltagestabilityregion} shows that the reduced power flow Jacobian $\varphi_u'$ is an effective tool for the voltage stability analysis. 
In particular, i) shows that studying the reduced power flow Jacobian is completely equivalent to studying the original power flow Jacobian, while
ii) provides a more precise characterization of the region where the grid voltages are stable.

\section{Voltage stability monitoring}
\label{sec:margin}

In this section we first propose an approximation of the determinant of the reduced power flow Jacobian, and then, based on this approximation, we propose a voltage stability index to quantify the distance of the power system from voltage collapse.

\subsection{Determinant approximation}

In Fig.~\ref{fig:datajacobian} we represent the numerical value of $\varphi_u'$ for two levels of loadability of the IEEE test feeder described in Section~\ref{sec:numerical}.
In the left panel, the operating point of the system is close to the flat voltage solution, while in the right panel, the grid is operated close to a loadability limit.
\begin{figure}[tb]
	\begin{center}
		\includegraphics[width=0.49\columnwidth]{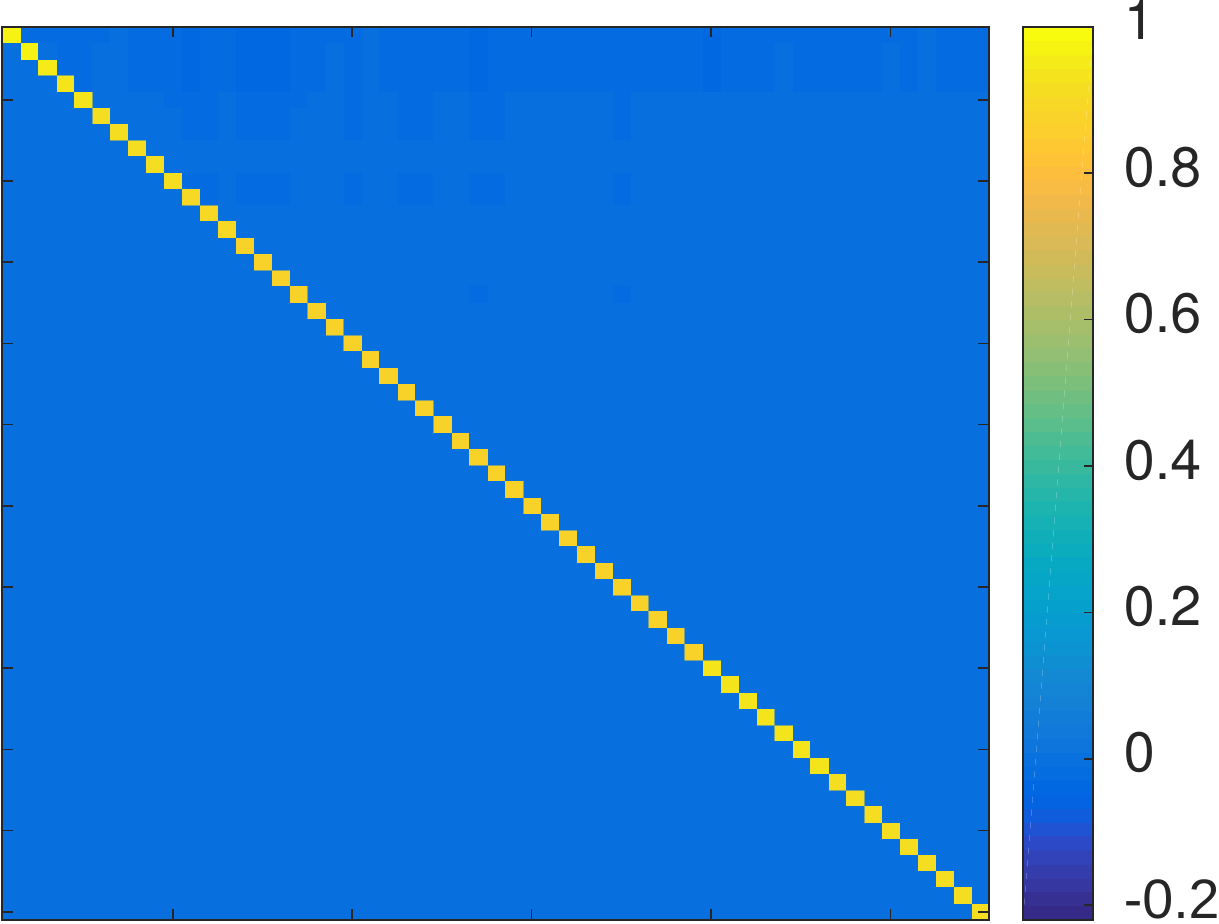}\hspace{\stretch{1}}
		\includegraphics[width=0.49\columnwidth]{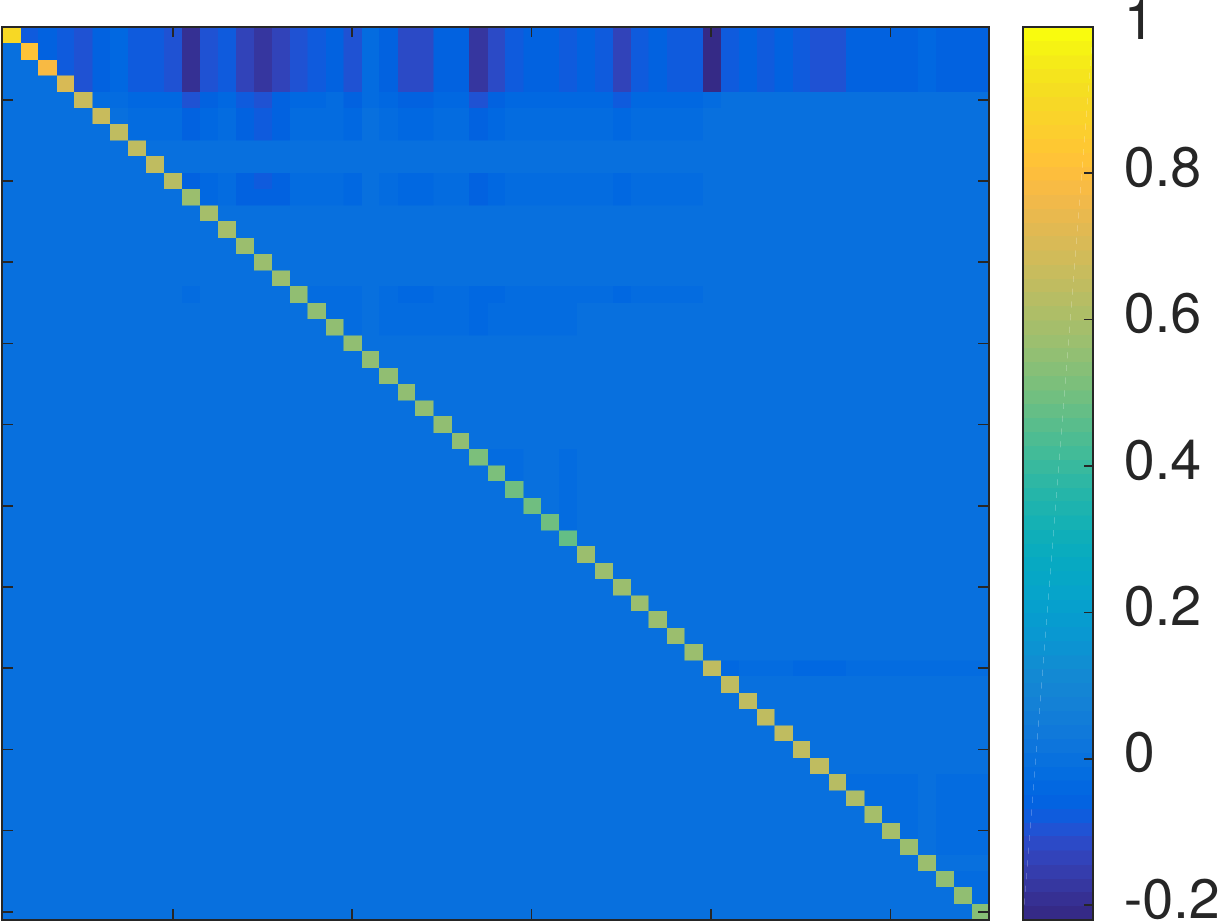}
	\caption{The value of the elements in the reduced power flow Jacobian for two levels of loadability.}
	\label{fig:datajacobian}
	\end{center}
\end{figure}

Direct inspection of the reduced power flow Jacobian $\varphi_u'$ shows that, for realistic parameter values and operating conditions, its off-diagonal elements (and in particular its lower-diagonal elements) are significantly smaller than the diagonal elements.
The approximation proposed in this paper consists in ignoring them.

The diagonal elements of $\varphi_u'$ are equal to 
\begin{equation}
\varphi_{u,jj}' = v_i - 2p_{ij}r_{ij} - 2q_{ij}x_{ij} - 2\ell_{ij}(r_{ij}\overline{r_{0i}} + x_{ij}\overline{x_{0i}})
\label{eq:diagelems}
\end{equation}
where $i = \pi(j)$, while $\overline{r_{0i}}$ and $\overline{x_{0i}}$ are the sum of the resistances (respectively, of the inductances) of the lines connecting node $0$ to node $i$.

By ignoring the off-diagonal elements, an approximation of $\det(\varphi_u')$ is obtained as the product of the elements on the diagonal defined in \eqref{eq:diagelems}:
\begin{equation}
	{\det}_\text{approx} = \prod_{j \in \{1,...,n\}} \varphi_{u,jj}'.
	\label{eq:detapproxprod}
\end{equation}

\begin{remark}
The index ${\det}_\text{approx}$ in \eqref{eq:detapproxprod} can be regarded as a natural generalization of the well-known voltage stability index for 2-bus networks (i.e., $n=1$) to arbitrary tree networks.
To see this, recall that for a 2-bus network
at the loadability limit, the magnitude of the load voltage is equal to the magnitude of the voltage drop on the line \cite{Cutsem1998}, i.e. 
\begin{equation}
v_1 = |\xi |^2, \quad \delta = \frac{1}{u_0^*} (r_{01} + \jay x_{01})(p_{01} - \jay q_{01})
\label{eq:vancutsemcondition}
\end{equation}
where $\xi$ and $u_0$ are the complex voltage drop on the line and the complex voltage of bus 0, respectively, and $u_0^*$ is the complex-conjugate of $u_0$.

By expanding $v_1$ as $|u_0 - \xi|^2$, condition \eqref{eq:vancutsemcondition} can be rewritten as $|u_0|^2 - 2 \realpart(u_0^* \xi) = 0$, and therefore, using the definition of $\xi$, as
$$
|u_0|^2 - 2 \realpart[(r_{01} + \jay x_{01})(p_{01} - \jay q_{01})] = 0.
$$
which is identical to $v_0 - 2p_{01}r_{01} - 2q_{01}x_{01} = 0$. 

This last quantity can also be written in terms of the reduced Jacobian $\varphi_u'$ and its approximation ${\det}_\text{approx}=\varphi_{u,11}'$ as
\begin{equation}
0=v_0 - 2p_{01}r_{01} - 2q_{01}x_{01} = \varphi_{u,11}' = \det(\varphi_u').
\label{eq:twobuscondition}
\end{equation}
The latter quantity \eqref{eq:twobuscondition} goes to zero at the loadability limit.

Thus, for $n=1$, our expression \eqref{eq:detapproxprod} does not introduce any approximation and recovers the well-known 2-bus condition. For an arbitrary tree network, ${\det}_\text{approx}$ is a natural generalization equal to the product of $n$ terms \eqref{eq:diagelems} similar to $\varphi'_{u,11}$, where in each of them an additional component accounting for the losses also appears. Each of these terms \eqref{eq:diagelems} corresponds to one edge of the network, therefore we have that ${\det}_\text{approx}=0$ when at least one $\varphi_{u,jj}'$, $j \in \{1,\ldots,n\}$, is equal to zero.
\end{remark}

In Section~\ref{sec:accuracy} and \ref{sec:numerical} we will analytically and numerically confirm that ignoring the off-diagonal elements results in a highly accurate approximation of the determinant in the voltage stability region.

\subsection{Voltage stability index}
\label{ssec:vsi}

Based on Theorem~\ref{thm:voltagestabilityregion}, the voltage stability region is defined as the region where $\det(\varphi_u') > 0$.
The numerical value of $\det(\varphi_u')$ provides a quantitative assessment of voltage stability, as larger values determine increased robustness with respect to parametric variation and fluctuations, and therefore encodes the \emph{distance} from voltage collapse (the boundary of the region where $\det(\varphi_u') > 0$).

In order to be useful for the practical assessment of voltage stability, a voltage stability index should take the same value for grids whose voltage stability is identical.
The following example provides some insight on this specification.

\begin{example}
Suppose we have a linear distribution feeder with $n+1$ nodes, where only node $1$ has positive power demand. Therefore $p_{01}>0$ and/or $q_{01} >0$ (so, $\ell_{01} > 0$), while for the other nodes $j \in \{2,...,n\}$, $p_{\pi(j) j} = q_{\pi(j) j} = \ell_{\pi(j) j} = 0$, and $v_1 = v_2 = ... = v_n$.

\begin{center}
\vspace{2mm}
\includegraphics[width=0.7\columnwidth]{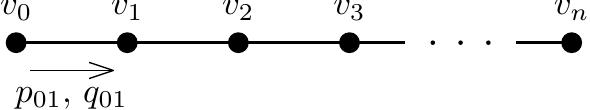}
\end{center}

The reduced power flow Jacobian is
\begin{equation*}
	\varphi_u' = 
	\begin{bmatrix}
		v_0-2p_{01}r_{01}-2q_{01}x_{01} & * & * & ... & * \\
		0 & v_1 & 0 &... & 0 \\
		0 & 0 & v_2 &...& 0 \\
		... & ... & ... & ... & ... \\
		0 & 0 & 0 & ... & v_{n-1}
	\end{bmatrix}
\label{eq:vsiex}
\end{equation*}

and its determinant can be explicitly calculated as 
\begin{equation*}
\det(\varphi_u') = (v_0-2p_{01}r_{01}-2q_{01}x_{01}) v_1^{n-1}
\end{equation*}

Because there are no power flows on the lines connecting the nodes ${2,\ldots,n}$ to node $1$, this grid is operationally equivalent to the 2-bus network composed just by the nodes $0$ and $1$. 
However, differently from the 2-bus network, since $v_1 < v_0 = 1$, the determinant of the $(n+1)$-bus network decreases exponentially in the number of nodes $n$.
\end{example}

The intuition from the above example can be generalized to arbitrary networks.
Recall that the determinant of a matrix is equal to the product of its eigenvalues.
It can be verified from \eqref{eq:rpfj} that, in the flat voltage solution, all the eigenvalues of $\varphi_u'$ are equal to 1.
For increasing power demands, all the eigenvalues get closer to the origin.
Since the number of eigenvalues is equal to the size of $\varphi_u'$, and thus to the size of the grid, larger networks (even more general than the line considered in the example) are thus naturally associated to exponentially smaller determinants. 

Based on this observation, we propose the scaled and normalized determinant 
\begin{equation}
\text{VSI} := \frac {\ln(\det(\varphi_u'))}{n}
\label{eq:vsimin}
\end{equation}
as a \emph{voltage stability index}.

Following the determinant approximation proposed in \eqref{eq:detapproxprod}, we then define the \emph{approximate voltage stability index} 
\begin{equation}
\text{AVSI} := \frac{\ln\left(\det_\text{approx}\right)}{n}
= \frac{1}{n}\sum_{j=1}^n h_j
\label{eq:avsi}
\end{equation}
where
\[
h_j := \ln\left(\varphi_{u,jj}'\right)
\]
and the terms $\varphi_{u,jj}'$ are defined in \eqref{eq:diagelems}.

\section{Computation of the AVSI}
\label{sec:compvsia}

In order to discuss some computational aspects of the proposed AVSI, we remark that each term $h_j$ in \eqref{eq:avsi} is a function of state variables that can be measured at node $j$.
In fact, by manipulating \eqref{eq:diagelems}, $h_j$ can be expressed as
\begin{equation}
h_j = \ln \left(
	v_j - \ell_{ij} \big( 
	r_{ij}(2\overline{r_{0j}} - r_{ij}) + x_{ij}(2\overline{x_{0j}} - x_{ij})
	\right)
\big),
\label{eq:equivdiagelem}
\end{equation}
which is a function only of the voltage magnitude $v_j$ at bus $j$ and of the squared current magnitude $\ell_{ij}$ on the power line connecting $j$ to its parent $i$.
Moreover, each term $h_j$ is only function of the local line parameters $r_{ij}$, $x_{ij}$ and of the line parameters $\overline{r_{0j}}$, $\overline{x_{0j}}$, which represent the electric distance of node $j$ from node $0$.

\subsection{Scalable centralized computation with linear complexity}
\label{subsec:cencomp}

Consider the case in which the entire grid state is available for centralized computation. 
This could be the case, for example, in distribution grids or islanded microgrids whose operation is monitored from a centralized location, to which all sensors send their real-time measurement.
It is also the case of numerical simulations of a power distribution grid, in which a designer is interested in evaluating the voltage stability of a multitude of different simulated loading scenarios.

The computation of the determinant of the $n \times n$ power flow Jacobian $\varphi_{u}'$ requires a computation time which is polynomial, precisely $O\left(n^3\right)$ when done via standard LU factorization \cite{Serre2010}, in the number of buses of the grid.

On the other hand, computing the proposed AVSI amounts to simply evaluating the arithmetic mean of the terms $h_j$ for all $j \in \{1,...,n\}$.
As the computation of each term $h_j$ requires constant time, the computational complexity of the AVSI is linear $O\left(n\right)$ with respect to the number of buses of the grid, resulting in a scalable and computationally efficient method also for large scale networks.

\subsection{Distributed computation}

Consider the case in which the sensor at each bus is also equipped with some computational power, and sensors are able to exchange information with (possibly a subset of) other sensors via a communication channel. 
In such a scenario, the proposed AVSI can be computed without relying on a centralized computation unit, therefore achieving increased robustness, scalability of the communication resources, and flexibility in case of network reconfiguration.

Recall that each $h_j$ depends on strictly local state measurements at bus $j$, local line parameters and the electrical distance between the nodes $0$ and $j$. The parameters $\overline{r_{0j}}$ and $\overline{x_{0j}}$ can be assumed to be known, or can be obtained via online estimation procedures \cite{Timbus2007,Gu2012} in a plug-and-play fashion. As a consequence, each $h_j$ can be computed locally, at bus $j$, without knowing the measurements at other buses or the entire electrical topology of the network.

The AVSI is then the algebraic mean of these local terms, and can therefore be computed in a distributed way by initializing each sensor state to the value $h_j$, and then running an \emph{average consensus protocol}  (see \cite{Bolognani2010} for a characterization of the family of algorithms that can be cast into this general-purpose protocol).

Average consensus protocols can be designed and tuned to converge (exponentially but possibly also in finite time) to the average of each sensor initial state in the presence of sparse communication graphs, time varying communication, and communication delays (see \cite{Xiao2004,Nedic2010,Bullo2018}).

In the following, we report a possible distributed algorithm that computes the AVSI in a distributed manner. The algorithm is to be executed in parallel by all the nodes in the set $N\backslash\{0\}$.
Given a communication graph, we denote by $\mathcal{N}_j$ the communication neighbors of node $j$ and by $d_j$ the communication degree of node $j$, i.e., the cardinality of the set $\mathcal{N}_j$.

\begin{algorithm}
	\caption{Distributed Computation of AVSI}
	\begin{algorithmic}
		\STATE{1. Each node $j$ computes its initial value $h_j$.}
		\STATE{2. Each node $j$ calculates its degree $d_j$.}
		\STATE{3. Each node $j$ sends $d_j$ and $h_j$ to its neighbors $\mathcal{N}_j$.}
		\STATE{4. Each node $j$ computes the weights \\
		\qquad\qquad $w_{jk} = 1/(1+\text{max}(d_j,d_k))$, \ $\forall k \in \mathcal{N}_j$ \\
		\qquad\qquad $w_{jj} = 1-\sum_{k \in \mathcal{N}_j}w_{jk}$.}
		\STATE{5. Each node $j$ updates its value $h_j$ as \\
		   \qquad\qquad $h_j \leftarrow w_{jj} h_j + \sum_{k \in \mathcal{N}_j} w_{jk} h_k$.}
		\STATE{6. Return to point 2.}
	\end{algorithmic}
	\label{alg}
\end{algorithm}

Direct application of the technical results in \cite{Xiao2005} guarantees exponential convergence of this algorithm to the AVSI if the union of the time-varying communication graph is connected.

\subsection{Hierarchical decomposition and recursive computation}

Power distribution grid are hierarchically structured in different levels, from medium voltage supra-regional and regional distribution grids, to low voltage local distribution grids \cite{Sallam2011}.
These different levels are often monitored independently, sometimes by different operators. 
In the following, we show that the proposed AVSI can be computed in a recursive way on this hierarchical structure. 

To formalize this idea, we introduce the following abstraction.
Given a set of nodes $N'$, let $\mathcal P(N')$ be a partition of $N'$, that is a set of sets such that $\bigcup_{N''\in \mathcal P(N')} = N'$ and all sets in $\mathcal P(N')$ have empty pair-wise intersection.
We start by partitioning the set of load buses $N\backslash \{0\}$, and we proceed recursively until we obtain trivial partitions (i.e., individual buses), as in Fig.~\ref{fig:hierachical}.

\begin{figure}[tb]
	\centering
\includegraphics[width=0.9\columnwidth]{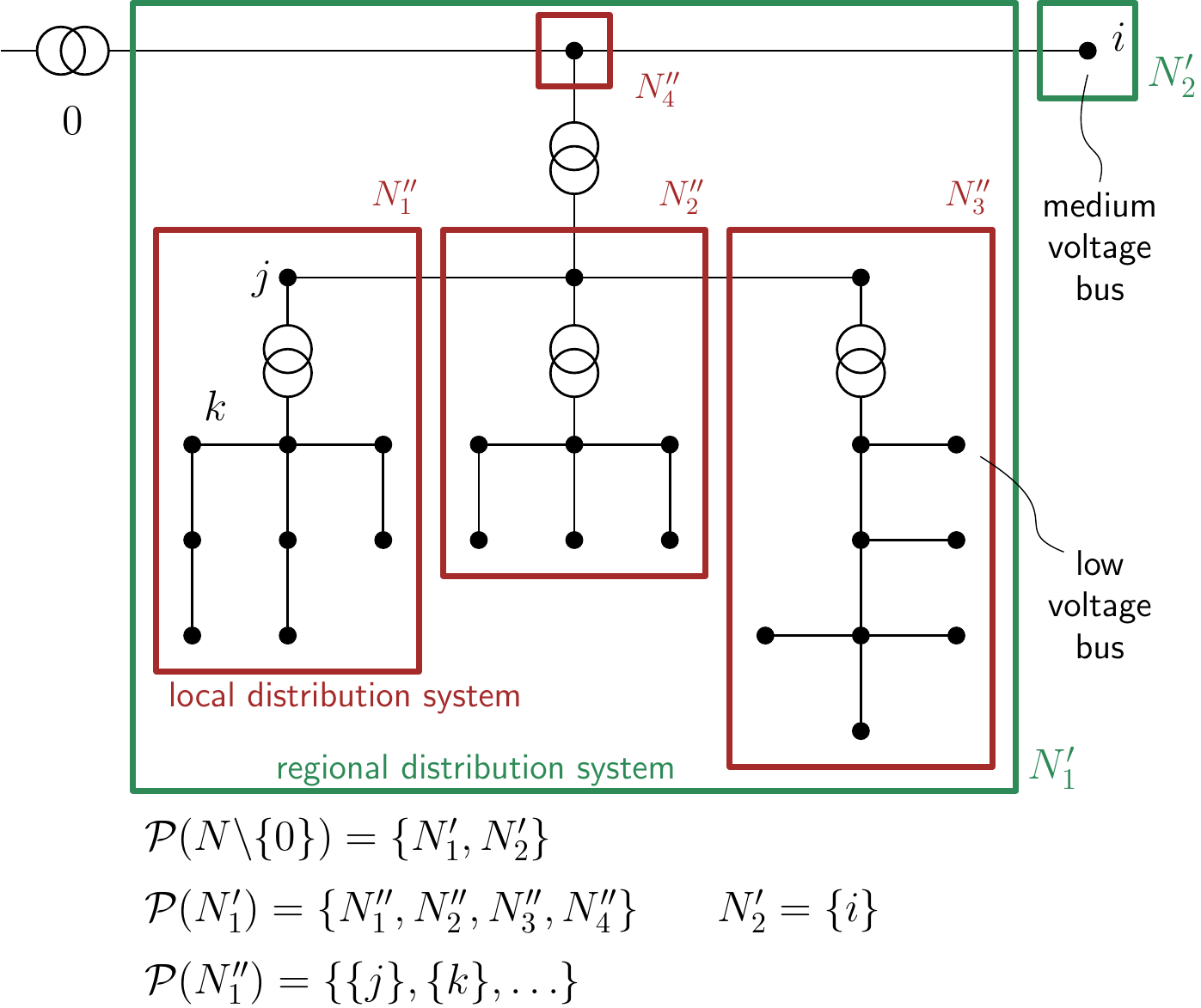}
\caption{An example of multi-level distribution network, with medium and low voltage sub-networks. The red and blue boxes show a possible hierarchical decomposition, in which nodes have been partitioned in areas. The proposed AVSI can be computed recursively on these partitions.}
\label{fig:hierachical}
\end{figure}

We consider the base case
\begin{align*}
n(N') &= 1 \\
H(N') &= h_j,
\end{align*}
when $N'$ is a single node $\{j\}$, and define, for any set $N'$ in the recursive partition, the following recursion step
\begin{align*}
n(N') &= \sum_{N'' \in \mathcal P(N')} n(N'') \\
H(N') &= \sum_{N'' \in \mathcal P(N')} H(N'').
\end{align*}

In other words, at each level of the hierarchy the quantities $H$ and $n$ are computed either based on the information coming from the operators of nested sub-grids (corresponding to a non-trivial subset $N''$) or by processing the sensor measurements (if $N''$ is the singleton subset $\{j\}$).

It is easy to show, using the properties of partitions, that the AVSI for the entire grid can be then recovered as
$$
\text{AVSI} = \frac{H(N\backslash\{0\})}{n(N\backslash\{0\})},
$$
where $N\backslash\{0\}$ represents the whole grid but the slack bus $0$.

In other words, each subnetwork can process the measurements coming from its sensors, and encode the necessary information in a compact piece of data that is then made available to the operator of the level immediately above.
Here, all these pieces of information are fused again, and forwarded upwards in the hierarchy. Ultimately, this procedure returns the AVSI for the entire grid.

\section{AVSI accuracy}
\label{sec:accuracy}

The accuracy of the proposed approximate voltage stability index $\text{AVSI}$ can be studied analytically under some extra assumptions on the operating regime of the distribution grid, namely under the assumption of mono-directional active and reactive power flows (from the slack node to the buses). 
Based on the adopted convention for the direction of the edges, this extra assumption can be formalized as follows.

\begin{assumption}
\label{ass:loads}
Active and reactive power flows on each line are nonnegative, i.e.,
\[
p_{ij}, q_{ij} \geq 0 \quad \forall (i,j) \in E
\]
\end{assumption}

In practical terms, having mono-directional power flows on the distribution grid corresponds to the most unfavorable case for voltage stability.

In the rest of this section we will make use of the following notation. We denote by $\varphi_{u,\text{diag}}'$ and $\varphi_{u,\text{off}}'$ the matrices that contain only the diagonal and off-diagonal elements of $\varphi_u'$, respectively. Moreover, we denote by 
\begin{equation}
\rho = \rho\left(\varphi_{u,\text{diag}}^{\prime -1}\varphi'_{u,\text{off}}\right)
\label{eq:rho}
\end{equation}
the spectral radius of $\varphi_{u,\text{diag}}^{\prime-1}\varphi'_{u,\text{off}}$, i.e. the maximum norm of its eigenvalues.

The results in this section build upon the mathematical theory of $Z$-matrices, $M$-matrices and $\tau$-matrices \cite{Mehrmann1984}:

\begin{definition} A matrix $A \in \mathbb{R}^{n \times n}$ is a 

\begin{itemize}
	\item $Z$-matrix if $A = \alpha I - B$, where $\alpha$ is a real number and $B$ is a nonnegative matrix.
	\item $M$-matrix if it is a $Z$-matrix and $\alpha \geq \rho(B)$.
	\item $\tau$-matrix if:
\begin{enumerate}
	\item[i)] Each principal submatrix of A has at least one real eigenvalue.
	\item[ii)] If $S_1$ is a principal submatrix of $A$ and $S_{11}$ a principal submatrix of $S_1$ then $\lambda_{min}(S_1) \leq \lambda_{min}(S_{11})$
	\item[iii)] $\lambda_{min}(A) \geq 0$
\end{enumerate}
where $\lambda_{min}$ denotes the smallest real eigenvalue.
\end{itemize}
\end{definition}

It can be verified by inspection of the sign pattern that the reduced power flow Jacobian $\varphi_{u}'$ is a $Z$-matrix for all operating points satisfying Assumption~\ref{ass:loads}.

In the following theorem, we present the result on the quality of the proposed approximate voltage stability index.

\begin{theorem}
In a power distribution network described by the relaxed branch flow model, with one slack bus and $n$ PQ buses, satisfying Assumption~\ref{ass:loads}, in the voltage stability region we have:
\label{thm:mainresult}
\begin{equation}
 \text{VSI} \leq \text{AVSI} \leq \text{VSI}  - \rho \, \ln(1 - \rho)
\label{eq:mainresult}
\end{equation}
where $\rho$ is defined in \eqref{eq:rho}.
\end{theorem}

\begin{proof}
We begin by proving that $\rho(\varphi_{u,\text{diag}}^{\prime -1}\varphi'_{u,\text{off}})<1$.

First notice that $\varphi_{u,\text{diag}}'$ is positive definite since $\varphi_{u,\text{diag}}' = I$ in the flat voltage solution and $\det(\varphi_{u,\text{diag}}^{\prime}) > 0$ in the voltage stability region. Therefore $\varphi_{u,\text{diag}}^{\prime -1}$ is well-defined. Now, since $\varphi_{u}^{\prime} = \varphi_{u,\text{diag}}^{\prime}(I + \varphi_{u,\text{diag}}^{\prime -1}\varphi_{u,\text{off}}^{\prime})$, we have that $\det(\varphi_{u}^{\prime}) = \det(\varphi_{u,\text{diag}}^{\prime})\det(I + \varphi_{u,\text{diag}}^{\prime -1}\varphi_{u,\text{off}}^{\prime})$.
In the flat voltage solution, $\varphi_{u,\text{diag}}^{\prime -1}\varphi_{u,\text{off}}^{\prime} = \mathbf{0}$ and in a loadability limit, $\det(I + \varphi_{u,\text{diag}}^{\prime -1}\varphi_{u,\text{off}}^{\prime}) = 0$. Thus, the power grid becomes unstable when an eigenvalue of $\varphi_{u,\text{diag}}^{\prime -1}\varphi_{u,\text{off}}^{\prime}$ arrives at $-1$.
Now, since $-\varphi_{u,\text{diag}}^{\prime -1}\varphi_{u,\text{off}}^{\prime}$ is non-negative, from the Perron-Frobenius Theorem \cite{Meyer2000} it has a positive real eigenvalue equal to the spectral radius $\rho(-\varphi_{u,\text{diag}}^{\prime -1}\varphi_{u,\text{off}}^{\prime})$. 
Therefore, $\varphi_{u,\text{diag}}^{\prime -1}\varphi_{u,\text{off}}^{\prime}$ has a negative real eigenvalue with magnitude equal to $\rho(\varphi_{u,\text{diag}}^{\prime -1}\varphi_{u,\text{off}}^{\prime})$. Hence, this is the eigenvalue that first arrives in $-1$. This implies that in the voltage stability region, $\rho(\varphi_{u,\text{diag}}^{\prime -1}\varphi_{u,\text{off}}^{\prime}) < 1$.

Via \cite[Theorem 1]{plemmons1977}, $\rho(\varphi_{u,\text{diag}}^{\prime -1}\varphi_{u,\text{off}}^{\prime})<1$ implies that $\varphi_{u}'$ is an $M$-matrix, and therefore the second inequality in \eqref{eq:mainresult} descends from \cite[Theorem 2.6]{Ipsen2011}.

Via \cite[Theorem 1]{Mehrmann1984}, every $M$-matrix is also a $\tau$-matrix. Therefore \cite[Theorem 4.3]{Engel1976} can be applied, obtaining the first inequality in \eqref{eq:mainresult}.
\end{proof}

Theorem~\ref{thm:mainresult} provides an upper bound on the VSI approximation error, given by $\rho \, \ln(1 - \rho)$ throughout the entire voltage stability region. 
A tighter bound can be conjectured, based on the observations in \cite{Ipsen2011}, which allow to replace $\rho \, \ln(1 - \rho)$ with  $(n_{\rho} / n) \, \rho \, \ln(1 - \rho)$, where $n_{\rho}$ is equal to the number of eigenvalues of $\varphi_{u,\text{diag}}^{\prime -1}\varphi_{u,\text{off}}^{\prime}$ whose magnitude is close to the spectral radius $\rho$. 
In our simulations we found that there is generally only one eigenvalue with magnitude close to the spectral radius, resulting in the following upper bound
\begin{equation}
 \text{AVSI} \leq \text{VSI} - \frac{1}{n} \, \rho \, \ln(1 - \rho)
\label{eq:conjecture}
\end{equation}

The lower bound $\text{VSI} \leq \text{AVSI}$ suggests that the determinant of the reduced power flow Jacobian may become zero before its approximation. 
However, as shown in the next section, the difference between the two indices is extremely small, making them effectively equivalent for practical purposes.

\section{Numerical validation}
\label{sec:numerical}

In this section we assess the quality of the proposed AVSI via numerical simulations on the modified IEEE 123-bus test feeder used in \cite{Bolognani2016,Wang2017}, and refer to \cite{github_approx-pf} for details.
This testbed contains an ensemble of balanced PQ loads with different power factors, connected via three-phase lines and cables with heterogeneous X/R ratio and shunt admittances.

\subsection{Quality of the approximation}
\label{subsec:qualityavsi}

\begin{figure*}[tb]
	\begin{center}
		\includegraphics[width=0.9\columnwidth]{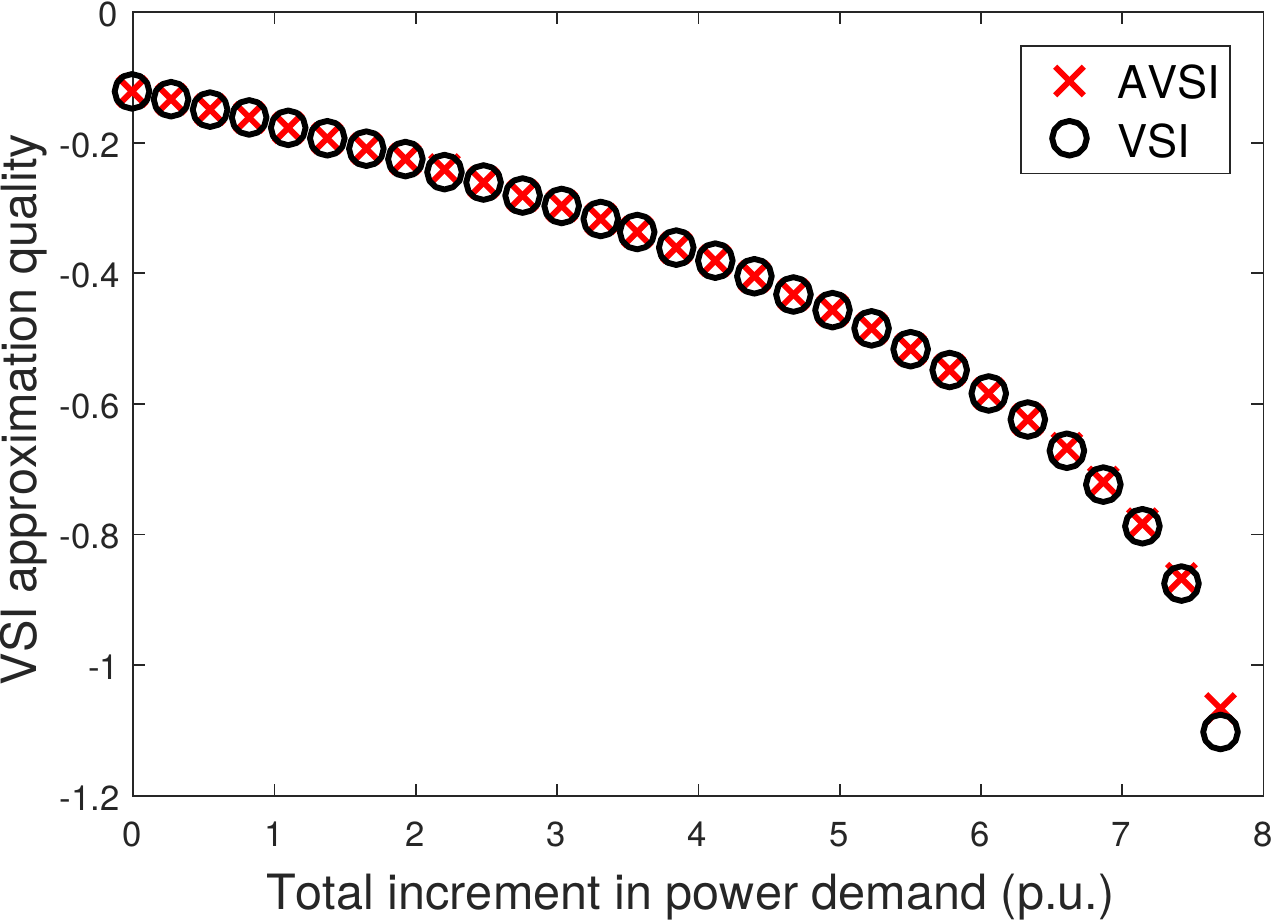}
		\qquad
		\includegraphics[width=0.9\columnwidth]{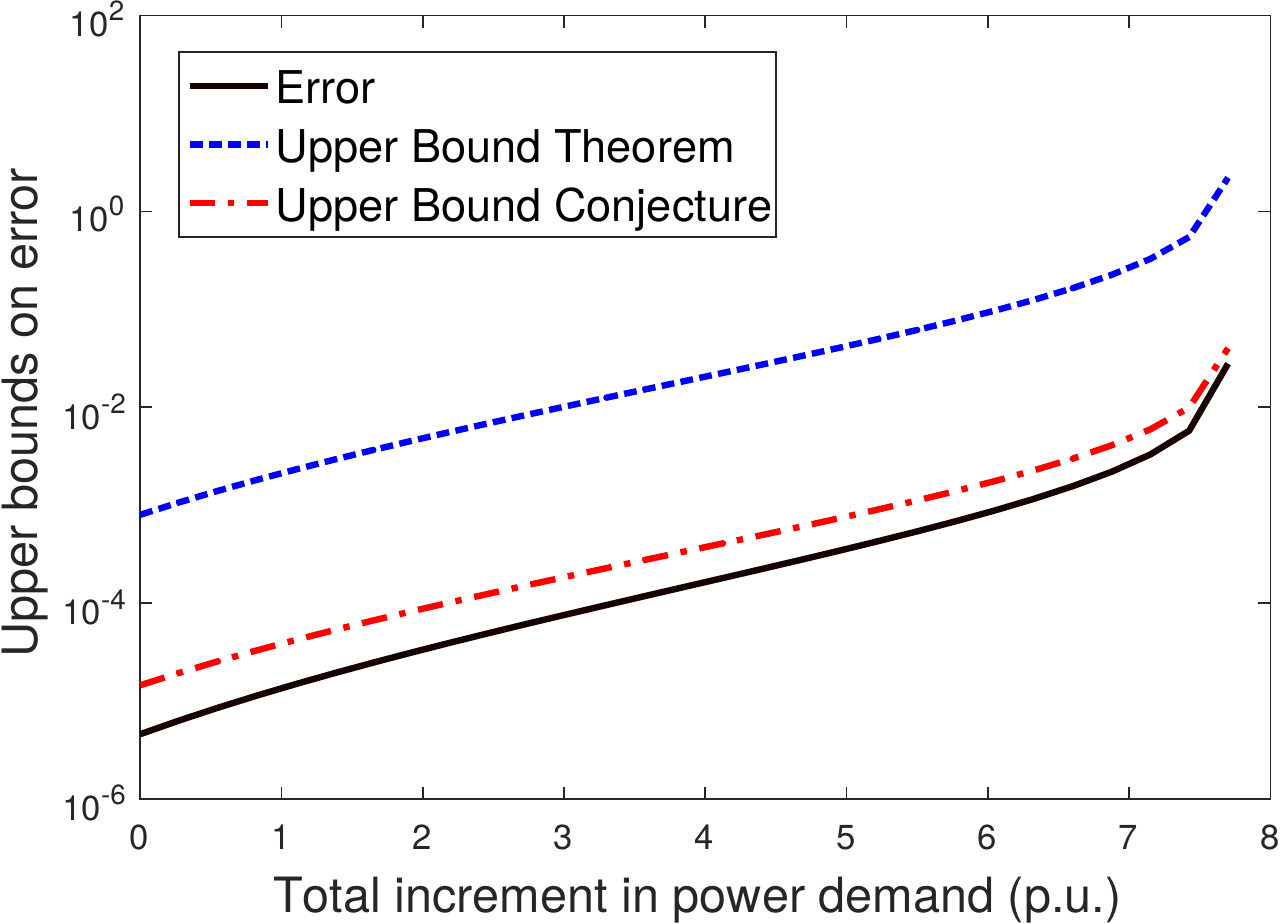}
	\end{center}
\end{figure*}

\begin{figure*}[tb]
	\begin{center}
		\includegraphics[width=0.9\columnwidth]{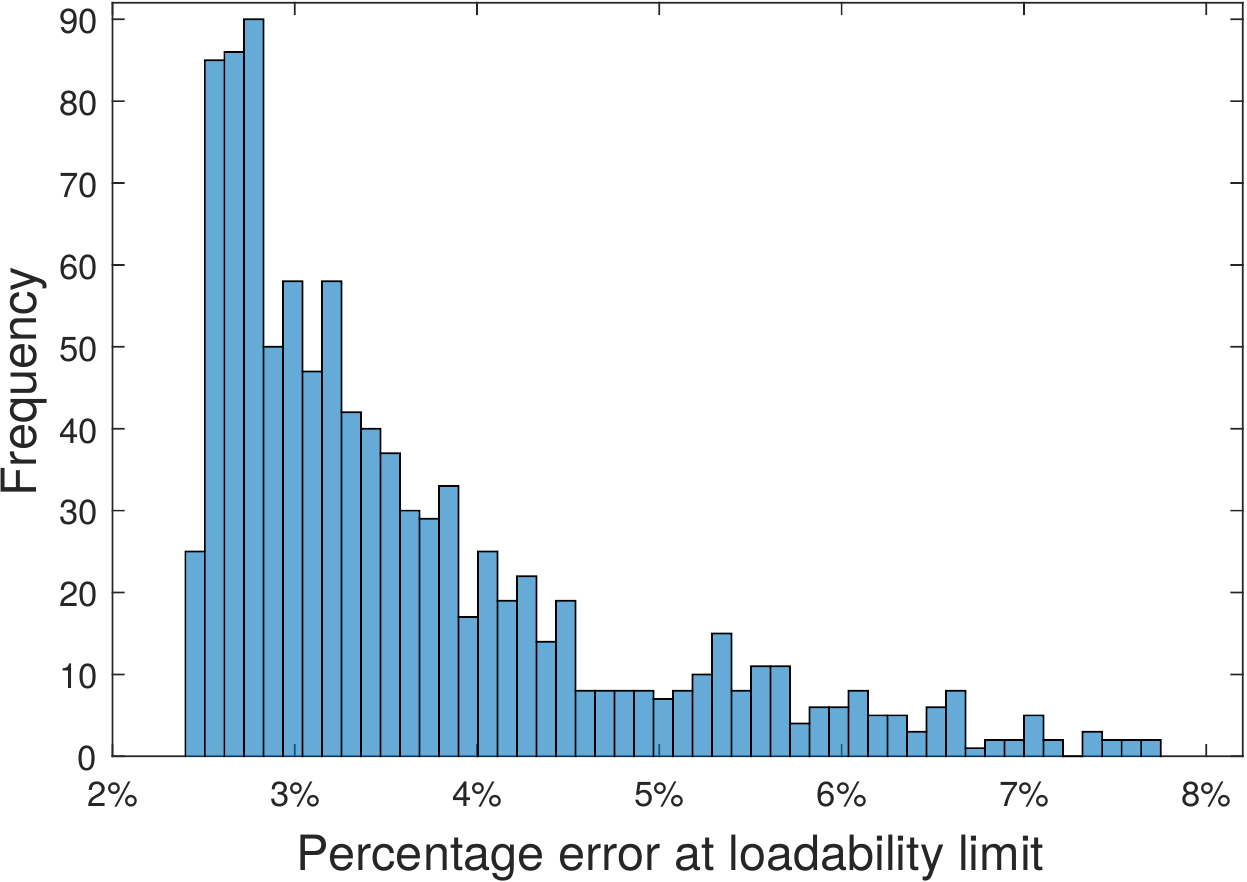}
		\qquad
		\includegraphics[width=0.9\columnwidth]{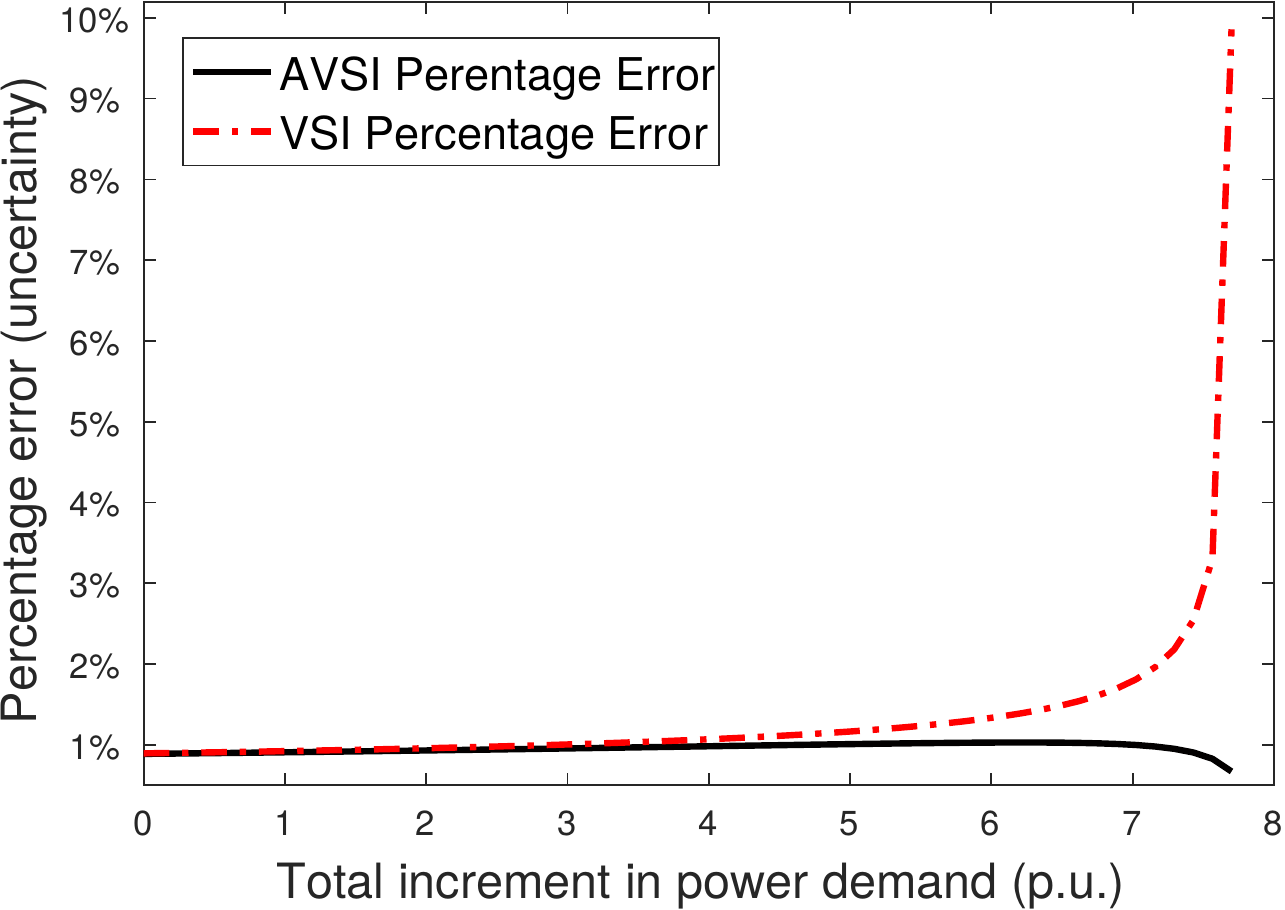}
	\caption{The top left subfigure shows the indices VSI and AVSI for a series of increasing demand levels, until voltage collapse. The top right subfigure shows the approximation error with the two proposed error bounds. The bottom left subfigure shows the percentage error histogram between VSI and AVSI at the loadability limit for $1000$ random loading scenarios. The bottom right subfigure shows the percentage error between the VSI computed with the exact line parameters, and the two indices AVSI and VSI, both computed with $25\%$ uncertainty in the line parameters.}
	\label{fig:avsi}
	\end{center}
\end{figure*}

To evaluate the quality of the approximation, we employ the continuation power flow method (as implemented in \cite{Zimmerman2011}) to obtain the power flow solutions of the grid for increasing power demands, starting from an operating point very close to the flat voltage solution, until the grid reaches a loadability limit (and therefore voltage collapse).
As the power demand increases, we compute both the VSI and the AVSI, and we evaluate the corresponding approximation error.

In the top left panel of Figure~\ref{fig:avsi} we represent the two indices VSI and AVSI when the power demand is increased uniformly across the entire grid.
Notice that while approaching the loadability limit, the negative slope of the VSI becomes very steep. 
Observe that the proposed approximation is almost exact up to very close to the loadability limit, where the AVSI becomes an upper bound to the exact VSI, as predicted by Theorem~\ref{thm:mainresult}.
The VSI approximation error and the two bounds presented in \eqref{eq:mainresult} and \eqref{eq:conjecture} are shown in the top right subfigure. Observe that the approximation error is monotonically increasing in the voltage stability region, starting from less than $10^{-5}$ at the base load and arriving to roughly $10^{-2}$ at the loadability limit.
Moreover, the conjecture \eqref{eq:conjecture} provides a very tight bound on the error.

We then repeat the same procedure for 1000 random loading scenarios.
Table~\ref{table} shows values of the VSI and of the AVSI (and the resulting approximation error denoted $\epsilon$) at the loadability limit, computed numerically via the continuation power flow method, while 
the bottom left subfigure of Figure~\ref{fig:avsi} illustrates its empirical distribution.

\begin{table}
\centering 
\footnotesize
\begin{tabular}{lccc}
\toprule
 Value & VSI & AVSI & Percentage error $\epsilon$ \\
\midrule
 Minimum & $-1.211$ & $-1.134$ & $2.42\%$ \\
 Average & $-1.106$ & $-1.065$ & $3.64\%$ \\
 Maximum & $-1.033$ & $-1.002$ & $7.74\%$ \\
\bottomrule 
\end{tabular}
\vspace{2mm}
\caption{VSI and AVSI at 1000 different loadability limits}
\label{table}
\end{table}

A threshold of approximately $-1$ seems to quantify very well the loadability limit of the grid for practical purposes.
At this point, where the determinant of the Jacobian is essentially zero (of the order of $e^{-n}$, according to the definition of the VSI), we are virtually at the point of voltage collapse.

It is important to notice that this threshold seems to be very insensitive to the specific loading pattern that is applied to the grid. 
Therefore, its scalar nature makes the index an effective indicator of distance from voltage collapse, as it is much simpler to identify (e.g., in simulations) what value of the index can be considered a safe voltage stability margin for a given grid, rather than trying to identify the region of voltage stable points in the high-dimensional space of complex bus power demands {\cite{Hiskens2001}}.

\subsection{Distributed generation}
\label{subsec:distribgen}

In this subsection we consider the presence of distributed generators (DGs).
The generators, modeled as constant power sources, account for nearly $20\%$ of the total number of buses and are uniformly spread throughout the network. A power factor of $0.9$ was applied to all the buses in the grid. 

We consider increasing DG penetration levels, defined as the ratio between the total apparent power at the DGs and the total apparent power at the loads \cite{Hoke2013}, from $10\%$ to $100\%$.
We consider $100$ random loading scenarios for each DG penetration level, and we consider increasing levels of loading (and, proportionally, generation). 
In Table~{\ref{tableDG}} we show the average values $\text{VSI}^{\text{AVG}}$ and $\text{AVSI}^{\text{AVG}}$ of, respectively, VSI and AVSI, at the loadability limit calculated via continuation power flow.
We also list the average and maximum value of the absolute and relative error ($\epsilon$ and $\epsilon_\%$, respectively) between the two indices at the loadability limit.

\begin{table}
\centering 
\footnotesize
\begin{tabular}{@{}ccccccc@{}}
\toprule
DG & $\text{VSI}^{\text{AVG}}$ & $\text{AVSI}^{\text{AVG}}$ & $\epsilon^{\text{AVG}}$ & $\epsilon^{\text{MAX}}$ & $\epsilon^{\text{AVG}}_\%$ & $\epsilon^{\text{MAX}}_\%$ \\
\midrule
 $10\%$ & $-1.06$ & $-1.02$ & $0.04$ & $0.07$ & $3.94\%$ & $8.49\%$ \\
 $20\%$ & $-1.07$ & $-1.03$ & $0.04$ & $0.07$ & $3.90\%$ & $8.29\%$ \\
 $30\%$ & $-1.06$ & $-1.02$ & $0.04$ & $0.07$ & $4.22\%$ & $8.42\%$ \\
 $40\%$ & $-1.06$ & $-1.01$ & $0.04$ & $0.07$ & $4.20\%$ & $7.92\%$ \\
 $50\%$ & $-1.04$ & $-1.00$ & $0.04$ & $0.07$ & $4.19\%$ & $7.95\%$ \\
 $60\%$ & $-1.04$ & $-1.00$ & $0.04$ & $0.07$ & $4.08\%$ & $9.07\%$ \\
 $70\%$ & $-0.97$ & $-0.93$ & $0.04$ & $0.07$ & $4.58\%$ & $10.51\%$ \\
 $80\%$ & $-0.82$ & $-0.78$ & $0.04$ & $0.07$ & $4.99\%$ & $9.84\%$ \\
 $90\%$ & $-0.62$ & $-0.58$ & $0.04$ & $0.07$ & $6.52\%$ & $13.74\%$ \\
 $100\%$ & $-0.35$ & $-0.30$ & $0.04$ & $0.07$ & $12.34\%$ & $18.51\%$ \\
\bottomrule 
\end{tabular}

\vspace{2mm}
\caption{Approximation error for different values of DG penetration}
\label{tableDG}
\end{table}

Notice that up to a penetration level of $60\%$ the indices VSI and AVSI at the loadability limit remain slightly below $-1$.
After this level, their values start increasing, up to at $-0.35$ and $-0.3$, respectively.
The reason for this phenomenon is apparent from \eqref{eq:diagelems}: the terms $\varphi_{u,jj}'$ corresponding to power lines that support a reverse power flow (towards the substation, i.e., $p_{ij}<0$ and/or $q_{ij}<0$) become larger as distributed generation becomes more significant.
Such information should be used when designing a practical threshold for the voltage stability index.

Notice moreover that the analysis of the approximation error proposed in Section~\ref{sec:accuracy} does not apply when reverse power flows are present, which is often the case when distributed generation is significant.
The analysis reported in Table~\ref{tableDG} shows however that the approximation error remains very small and constant (in absolute value) for all penetration levels.

\subsection{Uncertainty and robustness}

Up to this point the analysis has been carried out assuming that the feeder parameters contained in the vectors $r$ and $x$ are known and fixed.
However, in real-life scenarios, the values of these quantities may contain significant amount of uncertainties.
Consequently, in this subsection we consider up to $25\%$ uncertainty in the calculation or measurement of the line parameters and we show that the AVSI is a much more robust index compared to the VSI.
Based on the analytical expression of the AVSI, we identified the worst-case uncertainty as the case in which all power line impedances are over-estimated compared to their real value.

In the bottom right subfigure of Fig.~\ref{fig:avsi} we represent the percentage error between the VSI with uncertainty and the exact VSI, as well as the percentage error between the AVSI with uncertainty and the exact VSI.
As can be seen, the AVSI percentage error remains fairly steady at less than $1\%$ throughout the entire voltage stability region.
Meanwhile, the VSI percentage error drastically increases to almost $10\%$ close to the loadability limit.
We then considered 100 random loading scenarios and compute the same percentage errors for both AVSI and VSI at the loadability limits. While the the average and the maximum AVSI percentage errors are $0.70\%$ and $0.73\%$, respectively, the average and the maximum VSI percentage errors are $8.64\%$ and $9.85\%$, respectively.
This suggests that when such uncertainty is present, the index AVSI is much more robust than the VSI, and provides a better approximate of the correct stability index.

\section{Conclusions}
\label{sec:conclusions}

In this paper we considered the problem of assessing the voltage stability of a power distribution grid in a given operating point, based on a full observation of the state.
We propose an index that quantifies the distance from voltage collapse based on an accurate approximation of the determinant of the Jacobian of the power flow equations. 
The proposed index can be evaluated efficiently even for large networks, as it is suited to scalable, distributed, and hierarchical computation.
This, together with its numerical robustness with respect to parametric uncertainty of the grid model, makes the proposed index an effective solution for real-time monitoring of smart power distribution grids, for the assessment of voltage stability in large-scale randomized simulations, and also as a penalty function to include voltage stability constraints in optimal power flow programs.

Possible future developments of this methodology include the extension to more general classes of distribution grids (e.g., unbalanced and with voltage-regulated buses) and the derivation of upper bounds on the approximation error that can be evaluated a priori based on the grid parameters and that are valid also in the case of flow reversal.

%\bibliographystyle{IEEEtran}
%\bibliography{../voltagestability} 
\input{bibliography.bbl}

\begin{IEEEbiography}%
[{\includegraphics[width=1in,height=1.25in,clip,keepaspectratio]{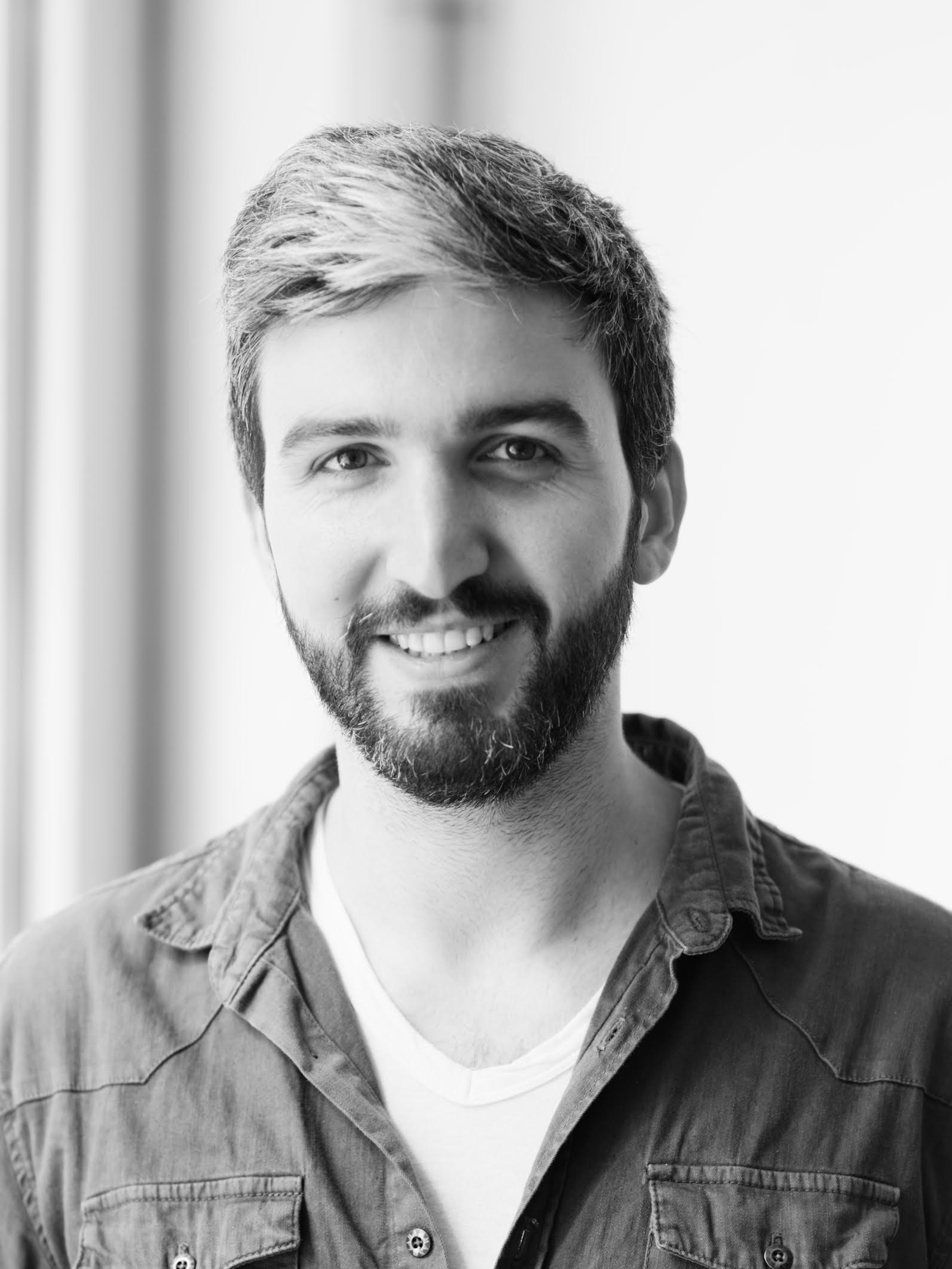}}]%
{Liviu Aolaritei} received the M.S. degree in Robotics, Systems and Control from ETH Zurich, Switzerland, in 2017, and the B.S. degree in Information Engineering from the University of Padova, Italy, in 2014. He was a visiting researcher at the Massachusetts Institute of Technology, USA, in 2017, and an intern in the ABB Corporate Research Center in Baden-D\"attwil, Switzerland, in 2016. He is currently a PhD student in the Automatic Control Laboratory at ETH Zurich.
\end{IEEEbiography}

\begin{IEEEbiography}%
[{\includegraphics[width=1in,height=1.25in,clip,keepaspectratio]{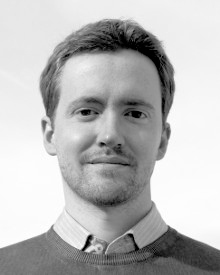}}]%
{Saverio Bolognani} received the B.S. degree in Information Engineering, the M.S. degree in Automation Engineering, and the Ph.D. degree in Information Engineering from the University of Padova, Italy, in 2005, 2007, and 2011, respectively. In 2006-2007, he was a visiting graduate student at the University of California at San Diego. In 2013-2014 he was a Postdoctoral Associate at the Laboratory for Information and Decision Systems of the Massachusetts Institute of Technology in Cambridge (MA). He is currently a Senior Researcher at the Automatic Control Laboratory at ETH Zurich. His research interests include the application of networked control system theory to smart power distribution networks, distributed control, estimation, and optimization, and cyber-physical systems.
\end{IEEEbiography}

\begin{IEEEbiography}%
[{\includegraphics[width=1in,height=1.25in,clip,keepaspectratio]{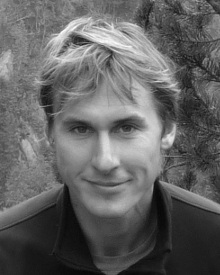}}]%
{Florian D\"{o}rfler} is an Assistant Professor at ETH Zurich. He received his Ph.D. degree in Mechanical Engineering from the University of California at Santa Barbara in 2013, and a Diplom degree in Engineering Cybernetics from the University of Stuttgart in 2008. From 2013 to 2014 he was an Assistant Professor at the University of California Los Angeles. His students were finalists for Best Student Paper awards at the European Control Conference (2013) and the American Control Conference (2016). His articles received the 2010 ACC Student Best Paper Award, the 2011 O. Hugo Schuck Best Paper Award, the 2012-2014 Automatica Best Paper Award, and the 2016 IEEE Circuits and Systems Guillemin-Cauer Best Paper Award. He is a recipient of the 2009 Regents Special International Fellowship, the 2011 Peter J. Frenkel Foundation Fellowship, and the 2015 UCSB ME Best Ph.D. award.
\end{IEEEbiography}

\end{document}